	\documentclass[12pt,letterpaper]{article}
	
	\usepackage{graphics,epsfig,graphicx,color}
	\graphicspath{{/EPSF/}{../Figures/}{Figures/}}
	
	\usepackage{amssymb,latexsym}
	
	\usepackage{setspace}
	\usepackage{amsmath}
	
	\usepackage{mathrsfs,amsfonts}
	
	\usepackage{amsthm,amsxtra}
	
	\newcommand\eq[1] {(\ref{#1})}

	\newtheorem{lemma}{Lemma}[section]
	\newtheorem{theorem}{Theorem}[section]
	\newtheorem{corollary}{Corollary}[section]
	\newtheorem{question}{Question}
	
	\newcommand{\bfm}[1]{\mbox{\boldmath ${#1}$}}
	
	\newcommand{\beqa}{\begin{eqnarray}}
	\newcommand{\eeqa}[1]{\label{#1}\end{eqnarray}}
	\newcommand{\beq}{\begin{equation}}
	\newcommand{\eeq}[1]{\label{#1}\end{equation}}

	\newcommand{\Gd}{\delta}

	\newcommand{\GD}{\Delta}

	\newcommand{\BGv}{\bfm\nu}

	\newcommand{\CD}{{\cal D}}

	\def\Bx{{\bf x}}
	\def\By{{\bf y}}
	\def\Bz{{\bf z}}
	
	\def\BB{{\bf B}}

	\def\BE{{\bf E}}

	\def\BH{{\bf H}}

	\def\BM{{\bf M}}

	\def\B0{{\bf 0}}

	\def \ph {\varphi}
	\def \RR {{\mathbb R}}
	
	\def \ba {\begin{array}}
	\def \ea {\end{array}}
	\newtheorem {Thm} {Theorem} [section]

	\newtheorem {Pro} [Thm] {Proposition}
	\newtheorem {Adef} [Thm] {Definition}
	
	\newtheorem {Arem} [Thm] {Remark}
	
	\newtheorem {Aexa} [Thm] {Example}
	
	\newtheorem {Anot} [Thm] {Notation}

	\def \refe #1.{(\ref{#1})}
	\def \reff #1.{figure~\ref{#1}}
	\def \refs #1.{section~\ref{#1}}
	\def \refss #1.{subsection~\ref{#1}}
	\def \refD #1.{Definition~\ref{#1}}
	\def \refT #1.{Theorem~\ref{#1}}
	\def \refL #1.{Lemma~\ref{#1}}
	\def \refC #1.{Corollary~\ref{#1}}
	\def \refP #1.{Proposition~\ref{#1}}
	\def \refR #1.{Remark~\ref{#1}}
	\def \refE #1.{Example~\ref{#1}}
	\def \refN #1.{Notation~\ref{#1}}

	\newcommand{\eps}{\varepsilon}

	\newcommand{\Ker}{{\rm Ker}}
	\newcommand{\dist}{\mbox{dist}}
	
	\newif\ifPDF
	\ifx\pdfoutput\undefined
	    \PDFfalse
	\else
	    \ifnum\pdfoutput > 0
	        \PDFtrue
	    \else
	        \PDFfalse
	    \fi
	\fi
	
	\ifPDF
	    \usepackage{pdftricks}
	    \begin{psinputs}
	        \usepackage{pstricks}
	        \usepackage{pstcol}
	        \usepackage{pst-plot}
	        \usepackage{pst-tree}
	        \usepackage{pst-eps}
	        \usepackage{multido}
	        \usepackage{pst-node}
	        \usepackage{pst-eps}
	    \end{psinputs}
	\else
	    \usepackage{pstricks}
	\fi
	
	\ifPDF
	    \usepackage[debug,pdftex,colorlinks=true,
	    linkcolor=blue, bookmarksopen=false,
	    plainpages=false,pdfpagelabels]{hyperref}
	\else
	    \usepackage[dvips]{hyperref}
	\fi
	
	\usepackage{fancyhdr}
	
	\newenvironment{keywords}
	{\noindent{\bf Key words.}\small}{\par\vspace{1ex}}

	\title{On the active manipulation of EM fields in open waveguides}
	
	\author{
	    Richard Albanese\thanks{ADED Co., San Antonio, TX} and Daniel Onofrei\thanks{
	        Department of Mathematics,
	        University of Houston,
	        Houston, Texas 77004; onofrei@math.uh.edu
	    }
	}
	
\begin{document}
	
	\maketitle
	
	\tableofcontents
	
	
	\begin{abstract}
	In this paper we present the extension of the results proposed in \cite{Onofrei-S} and study the problem of active control of TM waves propagating in a waveguide. The main goal is to cancel, in a prescribed near field region, the longitudinal component of the electric field of an incoming TM wave while having vanishingly small fields near the waveguide boundary. The main analytical challenge is to design appropriate source functions for the scalar Helmholtz equation so that the desired cancellation effect will be obtained. We show the existence of a class of solutions to the problem and provide  numerical support of the results. Discussion on the feasibility of the proposed approach as well as possible design strategies are offered.
	\end{abstract}
	
	
	
	
	
	\section{Introduction}
	
	The technique of manipulating acoustic and electromagnetic fields
	in desired regions of space has advanced in the recent
	years, with fascinating applications, such as
	cloaking, the creation of illusions, secure remote communication,
	focusing energy, and novel imaging techniques. The development can
	be roughly classified into two categories. 
	
	One main approach controls fields in the regions of interest by
	changing the material properties of the medium in certain surrounding regions
	while a second approach actively manipulates (active control) specially designed sources (antennas). Thus, briefly, we distinguish passive from active methods by the use of sources in the latter approach.
	
	In \cite{Green1} (see also \cite{Gunther-pr}) the authors presented an early rigorous discussion of the passive manipulation of fields in the context of
	quasistatics cloaking (see also 
	\cite{Dol}, \cite{Post} and \cite{LaxN} where the invariance to a change of variables is fully explained and the transformed material are described). This work was later extended in \cite{Pendry} to the
	general case of passive manipulation of fields in the finite frequency regime (see also the review \cite{Chan2}
	and references therein). These passive strategies are now known as ``transformation optics''.
	The similar strategy in the context of
	acoustics was proposed in \cite{Green4} (see also \cite{Cummer} and the review
	\cite{Chan3} and references therein). The idea behind
	transformation optics/acoustics is the invariance of the
	corresponding Dirichlet to Neumann-map (boundary measurements map)
	considered on some external boundary with respect to suitable
	change of variables which are identity on the respective boundary.
	This invariance result implies that two different materials (the initial one
	and the one obtained after the change of variables is applied)
	occupying some region of space $\Omega$, will have the same
	boundary measurements maps on $\partial \Omega$ and thus be
	equivalent from the point of view of an external observer. This
	mathematical structure leads to important applications, such as cloaking,
	field concentrators (\cite{Gunther1}) or field rotators, illusion optics, etc. (see
	\cite{Chan2}, \cite{Chan3}, \cite{Green2}, \cite{Alu} and references therein), and sensor cloaking while maintaining sensing capability \cite{Gunther2}, \cite{ALU4}.

	
	Additional passive techniques proposed in the literature (other than transformation optics strategies) include, plasmonic
	designs (see \cite{Alu} and references therein), strategies based
	on anomalous resonance phenomena (see \cite{Mil1}, \cite{Mil2},
	\cite{Mil3}), conformal mapping techniques (see
	\cite{Ulf1},\cite{Ulf2}), and complimentary media strategies (see
	\cite{Chan}).
	
		Active designs for the manipulation of fields appear to have occurred initially in the 
	 	context of low-frequency acoustics 
	 	(or active noise cancellation). Especially notable are pioneer works of Leug \cite{Leug} (feed-forward control of sound) and Olson \& May \cite{Olson-May} (feedback control of sound). The reviews \cite{Tsynkov}, \cite{Peterson}, \cite{Fuller}, \cite{Peake} \cite{Elliot} provide detailed accounts of past and recent developments in acoustic active control.
	 	
	Active control strategies are based on Huygens principle. The
	{\bf{\textit{interior}}} active cloaking strategy proposed in \cite{Miller}
	employs active boundaries while the {\bf{\textit{exterior}}} active cloaking
	scheme discussed in \cite{OMV1}, \cite{OMV2}, \cite{OMV3},
	\cite{OMV4} (see also \cite{CTchan-num}) uses a discrete number of active sources
	(antennas) to manipulate the fields. The active exterior strategy
	for 2D quasistatics cloaking was introduced in \cite{OMV1}, 
	and based on \emph{a priori} information about the incoming field.
	Guevara Vasquez, Milton and Onofrei \cite{OMV1} constructively
	described how one can create an almost zero field external region
	while maintaining a very small scattering effect in the far field.
	The proposed strategy did not work for objects close to the
	antennas, it cloaked large objects only when they are far enough
	to the antenna (see \cite{OMV4}). Also, the method was not adaptable for three
	space dimensions. The finite frequency case was studied in the
	last section of \cite{OMV1} and in \cite{OMV3} (see also
	\cite{OMV4} for a recent review) where three active sources
	(antennas) were needed to create a zero field region in the
	interior of their convex hull while creating a very small
	scattering effect in the far field. The broadband character of the
	proposed scheme was numerically observed in \cite{OMV2}. A
	general approach, based on the theory of boundary layer potentials, is proposed in \cite{Onofrei-S} for the solution of the active manipulation of quasi-static fields with one active source (antenna). 
	Several authors proposed experimental designs of active cloaking schemes in various regimes, \cite{Eleftheriades1}, \cite{Eleftheriades2}, \cite{Du} and \cite{Ma}.
	In this paper, we extend the results presented in \cite{Onofrei-S} to the case of locally nulling TM modes propagating in an open circular waveguide. The mathematical problem is formulated in the context of the scalar Helmholtz equation and the feasibility study discusses the relevance of this analysis in the context of a real antenna placed inside the waveguide and carrying electric conduction currents. 
	
	The paper is organized as follows: In Section~\ref{SEC:Form} we
	introduce the main problem of local nulling a waveguide TM mode (Question \ref{QUST:Main}). Then, in Section \ref{SEC:Exist} we prove the mathematical existence of a class of solutions for Question \ref{QUST:Main} and in \ref{SEC:Num} we offer several numerical simulations to support our 
	results. Antenna feasibility considerations and possible design strategies are discussed in Section ~\ref{F}. 
	
	\section{Problem formulation}
	\label{SEC:Form}
	
	In this paper $B_r(\Bx)$ will denote the disk situated in a plane perpendicular on the $Ox$ axis centered in point $\Bx\in\RR^3$ and of radius $r>0$. 
	  Consider the waveguide $D$, given by $D=(-\infty, +\infty)\times B_R(\B0)$, i.e, a circular cylinder of cross 
	 sectional area $B_R(\B0)$ and of infinite extent in the $Ox$ direction. We further consider that $D$ is a hollow waveguide filled with air and that the walls of the waveguide are made of perfect conductors.
	 In reality the waveguide walls have some level of loss, and by using the results of \cite{Kirsch}, we see that the assumption of zero loss provides a very good approximation for the case of small losses. 
	Inside the cylinder we consider an antenna (fixed there by a transparent dielectric), occupying the region $D_a$ which is modeled as a coaxial cylinder with smooth enough boundary and small cross sectional area, (Figure \ref{FIG:Geo} describes a the geometry for an antenna with straight edges).
	
	The region of interest (control region) will be denoted by $D_c$, and it will be assumed to be an open domain with $D_c\subsetneq D\setminus{\bar D}_a$. 
	 In Section \ref{SEC:Num} we will assume $D_c$ to be an annular shell coaxial with and in the near field of the antenna $D_a$,  defined by $D_c=\{B_{\Gd+l_1+l_2}(\B0)\setminus B_{\Gd+l_1}(\B0)\}\times[-a,a]$, where $l_1,l_2,a$ are given positive reals (see also Figure \ref{FIG:Geo}). 
	 Let us next denote by $(\BE_{i},\BB_{i})$ a given incoming field propagating down the waveguide and by $E_{xi}$ the longitudinal component of its electric field. The general question we want to study is:
	
	\begin{question}\label{QUST:Main}
	Can we synthesize an active source (antenna) $D_a$ (to be placed inside the waveguide $D$ as in Figure \ref{FIG:Geo}), modeled as a
	 magnetic current $\BM(\Bx)$ supported on $\partial D_a$, such that the field generated by it, $(\BE_{s},\BB_{s})$, with $\BE_s=(E_{xs}, E_{ys}, E_{zs})$ and $\BB_s=(B_{xs}, B_{ys}, B_{zs})$, has the property that in the region of interest $D_c$, $E_{xs}$ almost cancels $E_{xi}$? (Note that in the sketch of Figure \ref{FIG:Geo} we only presented a cylindrical antenna with straight edges although in practice one will consider the antenna with a smooth surface, e.g., cylindrical shape with hemispherical edges).
	\end{question}
	
	We divide the general problem presented in Question \ref{QUST:Main} in two subproblems:
	
	\begin{itemize}
	
	\item{{\bf{\underline{Problem 1.}}}}  Show that there exists boundary data $E_b$, such that $u$, the solution of the exterior Dirichlet problem for the Helmholtz equation with $E_b$ specified on $\partial D_a$, has the property that in the region of interest $D_c$ it almost cancels $E_{xi}$ while having vanishingly small values on the boundary of the waveguide?\medskip
	
	\item{{\bf{\underline{Problem 2.}}}} Based on the result of Problem 1, can we synthesize an active source (antenna) $D_a$ (to be placed inside the waveguide $D$ as in Figure \ref{FIG:Geo}), modeled as a surface magnetic current $\BM(\Bx)$ supported on $\partial D_a$ such that $E_{xs}$, the longitudinal component of its electric field, will be very close to $u$ described above and thus will posses the desired control properties?
	
	\end{itemize}
	
	We mention that {\bf{\underline{Problem 1}}} above studies the existence of solutions for a class of exterior Helmholtz problems which, as will be proved in Section \ref{F}, are essential in showing the existence of solutions for Question \ref{QUST:Main}. On the other hand, {\bf{\underline{Problem 2}}} above studies the feasibility of the antenna synthesis question.
	In this paper we will focus on {\bf{\underline{Problem 1}}} above in Section \ref{SEC:Exist} and Section \ref{SEC:Num} and present a brief feasibility study in Section \ref{F}. A much more detailed discussion about the feasibility of the second step in the context of the general antenna synthesis problem will be soon presented in \cite{OA}. 
	\begin{figure}[!ht]
	    \centering
	    \includegraphics[width=3in, height=3in]{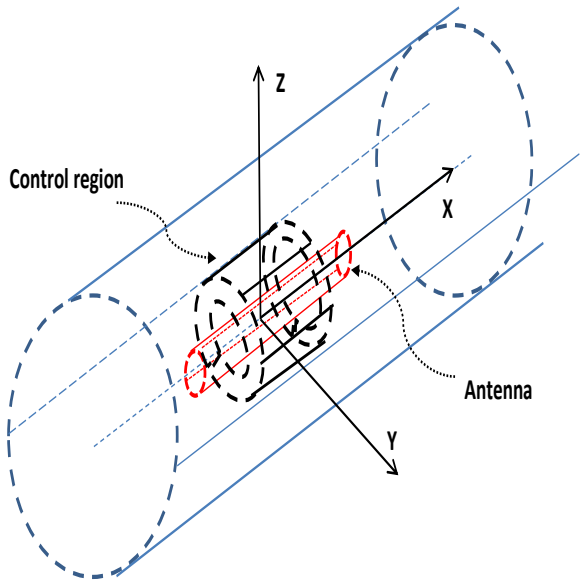}
	    \caption{The geometrical setting of Formulation A.}
	    \label{FIG:Geo}
	\end{figure}

		\section{Existence results}
		
		\label{SEC:Exist}
	
	A schematic illustration of the problem setting (where we assumed for simplification straight antenna edges) and various
	geometrical parameters are shown in Figure ~\ref{FIG:Geo}. Assume a general TM incoming electromagnetic wave $(\BB_i,\BE_i)$ with 
	$\BB_i=(0,B_{yi},B_{zi})$ and $\BE_i=(E_{xi},E_{yi},E_{zi})$. By recalling that each component of the electric and magnetic field satisfy the 3D scalar Helmholtz equation in their domain of analyticity, by using the boundary conditions on the surface of the waveguide, Problem 1 of Question~\ref{QUST:Main} reduces to the following mathematical formulation,
	
	\paragraph{\bf Mathematical equivalent of Problem 1}
	
	 Let $0<\eps\ll 1$ be fixed and assume $-\lambda^2$ is not a Neumann eigenvalue for the Laplacean in $D_a$. Consider an incoming TM wave $(\BE_{i},\BB_{i})$ and let $E_{xi}$ be the longitudinal component of its electric field. Then the equivalent mathematical formulation for Problem 1 of Question~\ref{QUST:Main} reads:
	 \vspace{0.2cm}

	 {\bf{\underline{Problem 1*.}}}
	 \noindent Find $E_b\in C(\partial D_a )$ such that there exists $u\in H^1(\RR^3\setminus {\bar D}_a)$ satisfying,
	 \beq\vspace{0.15cm}\left\{\vspace{0.15cm}\begin{array}{llll}
	 \GD u+\lambda^2u=0 \mbox{ in }\RR^3\setminus \overline{D}_a\vspace{0.15cm}\\
	 u=E_b \mbox{ on }\partial D_a\vspace{0.15cm}\\
	 \displaystyle\left(\frac{x}{|x|},\nabla u(x)) -iku(x)\right)=o\left(\frac{1}{|x|}\right),\mbox{ as }|x|\rightarrow\infty, \mbox{ uniformly for all directions $\displaystyle\frac{x}{|x|}$ }\nonumber \end{array}\right.,
	 \eeq{2}
	and
	\beq\Vert u\Vert_{L^2(\partial D)}\leq \eps,\;\Vert u+E_{xi}\Vert_{H^1({\bar{D}_c})}\le \eps.\eeq{2'}
	It is well known that, (see \cite{CoKr-Book83}), given $E_b\in C(\partial D_a)$, {\bf{\underline{Problem 1*.}}} has a unique solution in $H^1(\RR^3\setminus {\bar D}_a)$ for all $\lambda$ satisfying the above assumptions. 
	
	Thus, it remains to be proved that one can find a boundary data $E_b\in C(\partial D_a)$ so that the estimates \eqref{2'} are satisfied.
		For $L>0$ large enough consider $D_L$ sooth sub-domain of $\RR^3$ defined by \beq
			\{(x,y,z)\in D, |x|<L\}\subset D_{L}\subset \{(x,y,z)\in D, |x|<2L\}
			\eeq{3''}
	In what follows we make the assumption that the antenna, the region of control and the boundary of the waveguide are sufficiently well separated, i.e., there exist $D_{a'}, D_1$ smooth sub-domains of $\RR^3$ so that 
	
	\beq
	D_c\subsetneq D_1\subsetneq D_L\setminus{\bar{D}}_a.
	\eeq{3}  
	Consider also 
	\beq
	D_{a'}\subsetneq D_a
	\eeq{3'}
	
	In what follows we will also make the following assumption:
	\begin{equation}
	\label{ass-lambda}
	-\lambda^2 \mbox{ is not a Laplace Dirichlet eigenvalue in $D_1$. }
	\end{equation}
	Note that the set $D_1$ in \eqref{3} can be chosen so that \eqref{ass-lambda} is as well satisfied. We next introduce the following space $\Xi$,
	\begin{equation}
	\label{EQ:Parameters}
	    \Xi \equiv L^2(\partial D_1)\times L^2(\partial D_L).
	\end{equation}
	Then $\Xi$ is a Hilbert space with
	respect to the scalar product given by
	\begin{equation}\label{EQ:IP}
	(\varphi,\psi)_{\Xi}=\int_{\partial
	D_1}\varphi_1(\By){\overline \psi}_1(\By)  ds_{\By}+\int_{\partial
	D_L}\varphi_2(\By){\overline \psi}_2(\By) ds_{\By},
	\end{equation}
	for all $\varphi \equiv (\varphi_1,\varphi_2)$ and $\psi=(\psi_1,\psi_2)$ in
	$\Xi$ where $(\cdot)$ above, denotes the complex conjugate.
	
	The next Theorem is the main result of the Section and it shows the existence of a sequence of solutions for 
         problem \eqref{2}, \eqref{2'} such that 
         Problem 1 of Question \ref{QUST:Main} will have a positive answer.  We have:
	
	\vspace{0.2cm}
	
	\begin{theorem}
	\label{principala-teorema}
	
	Let $\lambda$ satisfying the hypothesis of ${\bf {Problem 1*}}$ and (\ref{ass-lambda}). Consider the following integral operator,
	$K:L^2(\partial D_{a'})\rightarrow \Xi$, defined as
	\begin{equation}\label{EQ:K}
	    Ku(\Bx,\Bz) =  (K_1u(\Bx), K_2u(\Bz)),
	\end{equation}
	for any $u(\Bx)\in L^2(\partial D_{a'})$, where
	\beqa
	K_1u(\Bx)&=&\int_{\partial D_{a'}}u(\By)\frac{\partial
	\Phi(\Bx,\By)}{\partial \BGv_{\By}}ds_{\By}, \mbox{ for }\Bx\in
	\partial D_1,\nonumber\\
	&&\nonumber\\
	K_2u(\Bz)&=&\int_{\partial D_{a'}}u(\By)\frac{\partial
	\Phi(\Bz,\By)}{\partial \BGv_{\By}}ds_{\By}, \mbox{ for }\Bz\in
	\partial D_L,
	\eeqa{8}
	with $\displaystyle\BGv_{\By}$ denoting the normal exterior to $\partial D_{a'}$ and where $\Phi(\Bx, \By)$
	represents the fundamental solution of the relevant Helmholtz operator, i.e.,
	\beq
	\Phi(\Bx,\By)=\frac{e^{i\lambda|\Bx-\By|}}{|\Bx-\By|}.
	\eeq{9}
	Then, the operator $K$ is compact and has a dense range. Moreover, for $\alpha_n\rightarrow 0$ (i.e., Tikhonov regularizers) the functions $v_n\in C(\partial D_a)$ defined by
	\begin{equation}
	\label{reg}
	v_n=(\alpha_n I+K^*K)^{-1}K^*v, \mbox{ with } v=(-E_{ix},0),\end{equation} satisfy that, for any $\epsilon<<1$, there exists an $N_\epsilon\in\mathbb N$ such that for any $n>N_\epsilon$ the functions $u_n$ given by,
	\begin{equation}
	\label{defun-0}
	u_n=\int_{\partial D_{a'}}v_n(\By)\frac{\partial \Phi(\Bx,\By)}{\partial \nu_\By}ds_\By\end{equation} satisfy,
	\begin{equation}
	\label{un-0}
	\left\{\begin{array}{lll}
	\vspace{0.2cm}\GD u_n +\lambda^2 u_n=0, \mbox{ in } D\setminus {\bar{D}}_{a'}& \vspace{0.2cm}\\
	\Vert u_n\Vert_{L^2(\partial D)}\leq \eps & \vspace{0.2cm}\\
	\Vert u_n+E_{xi}\Vert_{H^1( D_c)}\leq\eps.\vspace{0.2cm}\end{array}\right.
	\end{equation}
	\end{theorem} 
	\begin{proof}
	The next lemma presents a technical regularity results and it is needed in the proof of the Theorem.. Since the proof is classical we do not include it here but point out that the result can be obtained by adapting the proof in \cite{Kress-Book99}, (Section 3.4).

	\begin{lemma}\label{lemma-1}
	
	Let $\lambda\in \RR$ as in Theorem \ref{principala-teorema}.
	Let $f_i\in C(\partial D_1))$, and consider $v_i\in C^2(D_1)\cap C(\overline{D_1})$
	be the solution of the following interior Dirichlet problem,
	\beq
	\left\{\begin{array}{ll}
	\GD v_i +\lambda^2 v_i=0 \mbox{ in } D_1\vspace{0.15cm}\\
	v_i=f_i \mbox{ on } \partial
	D_1.\vspace{0.15cm}\end{array}\right.
	\eeq{6}
	Then we have,
	$$\displaystyle\Vert v_i\Vert_{H^1(D_c)}\leq
	C_i\Vert f_i\Vert_{L^2(\partial D_1)},$$
	where $C_i=C(\lambda,D_c,D_1)$.
	\end{lemma}

	We are now ready to present the proof of Theorem \eqref{principala-teorema}. Let us consider the integral operator,
	$K:L^2(\partial D_a)\rightarrow \Xi$, defined at \eqref{8}.
	
	The next Lemma given without proof, is a simple consequence of classical results in potential theory.
	
	\begin{lemma}
	\label{lemma-2}
	    The operator $K$ defined in~\eq{EQ:K} is a compact linear operator from $L^2(\partial D_{a'})$ to $\Xi$.
	\end{lemma}
	
	Let us introduce further the adjoint operator of $K$, i.e., the operator
	$K^*:\Xi\rightarrow L^2(\partial D_{a'})$ defined through the relation,
	\begin{equation}\label{EQ:K*}
	    (Kv,u)_{\Xi}=(v,K^*u)_{L^2(\partial D_{a'})},\ \forall u\in\Xi, v\in L^2(\partial D_{a'}),
	\end{equation}
	where $(\cdot,\cdot)_\Xi$ is the scalar product on $\Xi$ defined in~\eq{EQ:IP}
	and  $(\cdot,\cdot)_{L^2(\partial D_{a'})}$ denotes the usual scalar product
	in $L^2(\partial D_{a'})$ defined as a vectorial space over the complex field. We check, by simple change of variables and algebraic
	manipulations, that the adjoint operator $K^*$ is given by,
	\beq
	\displaystyle K^*u(\Bx)=\int_{\partial
	D_1}u_1(\By)\frac{\partial {\overline \Phi}(\Bx,\By)}{\partial
	\BGv_{\Bx}}ds_{\By}+\int_{\partial D_L}u_2(\By)\frac{\partial
	{\overline \Phi}(\Bx,\By)}{\partial \BGv_{\Bx}}ds_{\By},
	\eeq{10}
	for any $u=(u_1,u_2)\in\Xi$ and $\Bx\in\partial D_{a'}$.
	
	From the compactness and linearity of $K$ as given in Lemma~\ref{lemma-2},
	we conclude that the adjoint operator $K^*$ is compact as well.
	Furthermore, let us denote by $\Ker(K^*)$ the kernel (i.e., null space) of $K^*$.
	Then we have the following result.
	\begin{Pro}\label{Pro-1}
	 If $\psi=(\psi_1,\psi_2)\in \Ker(K^*)$ then $\psi\equiv(0,0)$ in $\Xi$.
	\end{Pro}
	\begin{proof}
	Let $\psi\in \Ker(K^*)$ and define \beq w(\Bx)=\int_{\partial
	D_1}{\overline\psi}_1(\By)\Phi(\Bx,\By)ds_{\By}+\int_{\partial
	D_L}{\overline\psi}_2(\By)\Phi(\Bx,\By)ds_{\By} \mbox{ for
	}\Bx\in\RR^3, \eeq{11} where the integrals exist as improper
	integrals for $\Bx\in \partial D_1\cup\partial
	D_L$. From $K^*\psi=0$ and \eq{10} we have that $w$
	satisfies the Laplace equation \beq \left\{\vspace{0.15cm}
	\begin{array}{ll}\GD w+\lambda^2 w=0, & \mbox{ in }
	D_{a'} \vspace{0.15cm}\\
	\displaystyle\frac{\partial w}{\partial\BGv_{\Bx}}=0,& \mbox{ on
	}\partial D_{a'}.
	\end{array}\right. .
	\eeq{12}
	We then conclude that
	\beq w=0 \mbox{ in }D_{a'}.
	\eeq{12'}
	Then, because by definition $w$ is a solution of Helmholtz equation in $D_L\setminus{\bar{D}}_1$, by analytic continuation we conclude that
	\beq
	w=0 \mbox{ in } D_L\setminus{\bar{D}}_1.
	\eeq{14}
	The next relations for $w$ are in fact the classical jump conditions
	for the single layer potentials with $L^2$ densities (see
	\cite{CoKr-Book98} and references therein). We have,
	\beqa
	&&\displaystyle\lim_{h\rightarrow 0^+}\int_{\partial
	D_1}|w(\Bx\pm
	h\BGv_{\Bx})-w(\Bx)|^2ds_{\Bx}=0\label{15-1}\\
	&&\nonumber\\
	&&\displaystyle\lim_{h\rightarrow 0^+}\int_{\partial
	D_L}|w(\Bx\pm h\BGv_{\Bx})-w(\Bx)|^2ds_{\Bx}=0 \label{15-2}\\
	&&\nonumber\\
	&&\displaystyle\lim_{h\rightarrow 0^+}\int_{\partial
	D_1}\left|2\frac{\partial w}{\partial\BGv_{\Bx}}(\Bx\pm
	h\BGv_{\Bx})-2\frac{\partial
	w}{\partial\BGv_{\Bx}}(\Bx)\pm\psi_1(\Bx)\right|^2ds_{\Bx}=0\label{15-3}\\
	&&\nonumber\\
	&&\displaystyle\lim_{h\rightarrow 0^+}\int_{\partial
	D_L}\left|2\frac{\partial w}{\partial\BGv_{\Bx}}(\Bx\pm
	h\BGv_{\Bx})-2\frac{\partial
	w}{\partial\BGv_{\Bx}}(\Bx)\pm\psi_2(\Bx)\right|^2ds_{\Bx}=0,
	\eeqa{15-4} where $\BGv_\Bx=\BGv(\Bx)$ denotes the exterior normal
	to $\partial D_L$ and $\partial D_1$ respectively and all
	the integral of the normal derivatives of $w$ exists as improper
	integrals. From \eq{14}, \eq{15-1} and \eq{15-2} we obtain that
	\beq w=0\mbox{ on } \partial D_L\cup\partial D_1.
	\eeq{16} 
	Next from \eqref{ass-lambda} and by uniqueness of the interior Dirichlet
	problem for $w$ on $D_1$ and \eq{16} we obtain \beq w=0
	\mbox{ in } \bar{D}_1. \eeq{17} From \eq{14}, \eq{17},
	and the two jump relations \eq{15-3}, we obtain that \beq \psi_1=0
	\mbox{ on } \partial D_1. \eeq{17'} Equation \eq{17'}
	used in the definition of $w$ given at \eq{11}, implies \beq
	w(\Bx)=\int_{\partial
	D_L}{\overline\psi}_2(\By)\Phi(\Bx,\By)ds_{\By}, \mbox{ for
	}\Bx\in\RR^3.\eeq{17''} 
	Relations \eq{14}, \eq{16}, and
	\eq{17} imply that \beq
	    w=0 \mbox{ in } {\bar D_L} .
	\eeq{18}
	From the interior jump condition given in \eq{15-4} together with \eq{18}
	we have that
	\beq
	\displaystyle \frac{\partial w}{\partial \BGv_\Bx} = -\frac{1}{2}{\overline\psi}_2 \mbox{ a.e. on }
	\partial D_L.
	\eeq{20} 
	Since $w$ is a solution of the Helmholtz equation in $D_L$, by using the Green representation theorem for $\Bx\in \RR^3\setminus {\overline D_L}$ we obtain,
	
	\begin{eqnarray}
	\label{w-at-infinity}
	0&=&\int_{\partial D_L}
	\frac{\partial w}{\partial \nu_\By}(\By)\Phi(\Bx,\By)ds_\By - 
	\int_{\partial D_L}
	w(\By)\frac{\partial \Phi(\Bx,\By)}{\partial \nu_\By}(\By)ds_\By\nonumber\\
	& &\nonumber\\
	&=&-\frac{1}{2} \int_{\partial D_L}
	{\overline\psi}_2(\By)\Phi(\Bx,\By)ds_\By\nonumber\\
	& & \nonumber\\
	&=& -\frac{1}{2} w(\Bx),
	\end{eqnarray}
	where we used  \eqref{17''},\eq{18} and \eq{20} in the equalities above.
	Thus from \eqref{w-at-infinity} and the jump conditions given at \eqref{15-4} we obtain $\psi_2=0$ on $\partial D_L$.
	\end{proof}
	
	Let us introduce the following space of functions,
	\[
	    U \equiv K(C(\partial D_{a'}) .
	\]
	It is clear that $U$ is a subspace of $\Xi$. Moreover we have,
	\begin{lemma}\label{LMMA:U}
	    The set $U\subset\Xi$ is dense in $\Xi$.
	\end{lemma}
	\begin{proof}
	    We first observe that the subspace $U\subset \Xi$ satisfies
	\beq
	    \overline{U} =\left(U^\perp\right)^\perp,
	\eeq{26} where here and further in the proof, for a given set
	$M\subset\Xi$, $\overline{M}$ and $M^\perp$ denote its closure and
	orthogonal complement respectively in the $L^2$ topology generated
	on $\Xi$ by the scalar product defined at \eq{EQ:IP}. Property
	\eq{26} is classic for subspaces in a Hilbert space (see
	\cite{B}). On the other hand we also have that \beq
	    U^\perp=\Ker(K^*).
	\eeq{27}
	Indeed let $\xi=(\xi_1,\xi_2)\in U^\perp$. Then, for all $\ph\in C(\partial D_{a'})$  we have,
	\beqa
	    0=(K\ph,\xi)_{\Xi}&\Leftrightarrow&(\ph,K^*\xi)_{L^2(\partial D_{a'})}=0\Leftrightarrow\nonumber\\
	&\Leftrightarrow& K^*\xi=0\Leftrightarrow \xi\in
	\Ker(K^*) .
	\eeqa{28}
	Properties \eq{26} and \eq{27} imply that
	\beq
	{\overline U}= \Ker(K^*)^\perp.
	\eeq{29}
	Proposition \ref{Pro-1} together with \eq{29} imply the density of $U$ in $\Xi$.
	\end{proof}
	We are now in the position to state and prove an essential result of the paper.
	\begin{lemma}
	\label{thm-1}
	 Let $D_1, D_{a'}$ be given as in \eqref{3} and \eqref{3'}. Let $v=(v_1,v_2)\in C^1({\bar D}_1)\times L^2(\partial D_L)$
	 be such that $v_1$ is a solution of Helmholtz equation in $D_1$.
	Define the double layer potential $\CD$ with density $\varphi\in L^2(\partial D_{a'})$ as,
	 $$ \displaystyle\CD \varphi(\Bx)=\int_{\partial
	 D_{a'}}\varphi(\By)\frac{\partial\Phi(\Bx,\By)}{\partial\nu_{\By}}ds_\By, \mbox{ for }\Bx\in \RR^3\setminus \bar{D}_{a'}.$$
	 Then $\CD:L^2(\partial D_{a'})\rightarrow C(\RR^3\setminus {\overline D}_{a'})$
	 is a continuous operator between $L^2(\partial D_{a'})$ and $C(\RR^3\setminus
	 {\overline D}_{a'})$ endowed with their natural topologies. Moreover, there exists a sequence $\{v_n\}\subset C(\partial D_a)$ such that
	 $$\CD v_n\rightarrow v_1 \mbox{ strongly in }H^1(D_c), \mbox {and }\; \CD v_n\rightarrow v_2 \mbox{ strongly in } L^2(\partial D_L).$$
	\end{lemma}
	\begin{proof}
	    We first observe that $v\in\Xi$.
	Then the definition of $U$ and Lemma~\ref{LMMA:U} imply that there exists a sequence $\{v_n\}\subset C(\partial D_a)$ such that
	\beq
	K(v_{n})\rightarrow v \mbox{ strongly in }\Xi.
	\eeq{29'}
	From the definition of the $\Xi$ topology and \eq{29'} we conclude that
	\beqa
	&&\Vert K_1v_{n}-v_1\Vert_{L^2(\partial D_1)}\rightarrow
	 0,\nonumber\\
	 &&\label{30'}\\
	&&\Vert K_2v_{n}-v_2\Vert_{L^2(\partial D_L)}\rightarrow
	 0.\nonumber
	\eeqa{30}
	Observe that, by definition, $K_1v_{n}$ (resp. $K_2 v_{n}$) is the restriction
	to $\partial D_1$ (resp. $\partial D_L$) of $\CD v_{n}$ (resp. $\CD v_{n}$)
	where $\CD$ was defined in the statement of the Theorem. From the properties of $\CD$,
	the hypothesis on $v_1,v_2$ and the regularity results of
	Lemma~\ref{lemma-1}, we conclude that
	\beqa &&\Vert
	\CD v_{n}-v_1\Vert_{H^1(D_c)}\leq C_1\Vert K_1 v_{n}-v_1\Vert_{L^2(\partial D_1)},\nonumber\\
	 &&\label{31}\\
	&&\Vert \CD v_{n}-v_2\Vert_{L^2(\partial D_L)}= 
	\Vert K_2v_n-v_2\Vert_{L^2(\partial
	D_L)},\nonumber\eeqa{31'} 
	 for some constant $C_1=C_1(\lambda,D_c,D_1)$. Finally from \eq{30'} and \eq{31} we obtain the
	statement of the Lemma.
	\end{proof}
	
	From Lemma \ref{thm-1} applied with $v=(E_{xi},0)$ we deduce that there exists a sequence $v_n\in C(\partial D_{a'})$ such that for any $\epsilon<<1$ there exists an $N_\epsilon\in\mathbb N$ such that for any $n>N_\epsilon$ the functions $u_n$ given by
	\begin{equation}
	\label{defun}
	u_n=\int_{\partial D_a}v_n(\By)\frac{\partial \Phi(\Bx,\By)}{\partial \nu_\By}ds_\By,\end{equation} satisfy \eqref{un-0}. Moreover, for $L$ large enough we have,
	\begin{equation}
	\label{un}
	\left\{\begin{array}{llll}
	\vspace{0.2cm}\GD u_n +\lambda^2 u_n=0, \mbox{ in } D\setminus {\bar D}_a& \vspace{0.2cm}\\
	u_n(\Bx)=E_b(\Bx)=\int_{\partial D_a}v_n(\By)\frac{\partial \Phi(\Bx,\By)}{\partial \nu_\By}ds_\By& \mbox{ for }\Bx\in \partial D_a\vspace{0.2cm}\\
	\Vert u_n\Vert_{L^2(\partial D)}\leq\epsilon\vspace{0.2cm}\\
	\Vert u_n+E_{xi}\Vert_{H^1( D_c)}\leq\epsilon.\vspace{0.2cm}\end{array}\right.
	\end{equation}
	 where we have used the decay of the double layer potential to show that $\Vert u_n\Vert_{L^2(\partial D\setminus \partial D_L)}=o(1)\mbox {with respect to $L\rightarrow\infty$}$. Observe that the functions $v_n$ in \eqref{defun} can be obtained for example through a Tikhonov regularization procedure and we have 
	$$v_n=(\alpha_n I+K^*K)^{-1}K^*v$$ for some $\alpha_n\rightarrow 0$.
	Thus, we showed that there exists a sequence of possible boundary data  $E_b$ described in \eqref{un} such that  the sequence $u_n$ satisfies the statement of Theorem \ref{principala-teorema}.
	\end{proof}

	\begin{Arem}\label{Rem1}
	We observe that one can easily adapt the proof of Theorem
	\ref{thm-1} to the case of two or more mutually
	disjoint regions of interest. Thus, following the same arguments
	as before, one will obtain a class of solutions for
	Question~\ref{QUST:Main} in this general context. 
	\end{Arem}

%
	%
	
	\begin{Arem}\label{Rem5}
	We also would like to point out that our results are immediately applicable to 2D or 3D acoustics thus answering the problem of active acoustic control as well.
	\end{Arem}
	%
	
	\section{Numerical support}
	
	\label{SEC:Num}
	
	In this Section we will consider $D_a=[-l,l]\times B_\delta(\B0)$ and assume $D_c$ to be an annular shell coaxial with and near the antenna $D_a$,  defined by $D_c=\{B_{\Gd+l_1+l_2}(\B0)\setminus B_{\Gd+l_1}(\B0)\}\times[-a,a]$, where $l_1,l_2,a$ are given positive reals to be specified in the numerical tests.
	In Figure \ref{FIG:Geo} we sketch the considered geometry. 
	We will next present several numerical simulations to support the theoretical results obtained above. First note that in this regime the electromagnetic waves propagating in the $x$ direction are given by:
	 \begin{equation}\label{f-gen-form}\left\{\begin{array}{ll}
	 \BE(x,y,z)=\BE_0(y,z)e^{i(\beta x-\omega t)}, & \\
	 \BB(x,y,z)=\BB_0(y,z)e^{i(\beta x-\omega t)},\end{array}\right.\end{equation}
	 for some $\beta$ and where $\BE_0=(E_x, E_y, E_z)$ and $\BB_0=(B_x, B_y, B_z)$.
	 After using \eqref{f-gen-form} in the Maxwell equations we arrive at the following,
	 
	 \begin{equation}
	 \label{transverse-fields}\left\{\begin{array}{llll}\vspace{0.2cm}
	 E_y=\frac{i}{\left(\frac{\omega}{c}\right)^2-\beta^2}\left(\beta\frac{\partial E_x}{\partial y}+\omega \frac{\partial B_x}{\partial z}\right), & \\
	 \vspace{0.2cm}E_z=\frac{i}{\left(\frac{\omega}{c}\right)^2-\beta^2}\left(\beta\frac{\partial E_x}{\partial z}-\omega \frac{\partial B_x}{\partial y}\right), & \\
	 \vspace{0.2cm}B_y=\frac{i}{\left(\frac{\omega}{c}\right)^2-\beta^2}\left(\beta\frac{\partial B_x}{\partial y}-\frac{\omega}{c^2} \frac{\partial E_x}{\partial z}\right), & \\
	 \vspace{0.2cm} B_z=\frac{i}{\left(\frac{\omega}{c}\right)^2-\beta^2}\left(\beta\frac{\partial B_x}{\partial z}+\frac{\omega}{c^2} \frac{\partial E_x}{\partial y}\right),
	 \end{array}
	 \right.
	 \end{equation}
	 where $c=\frac{1}{\sqrt{\mu\epsilon}}$ with $\mu,\epsilon$ denoting the permeability and permittivity of air. We also have that the longitudinal components of the fields $E_x$ and $B_x$ satisfy in $B_R(\B0)$ the following two-dimensional Helmholtz equation,
	 \begin{eqnarray}
	 \label{helmholtz-full-2D}
	 \frac{\partial^2}{\partial y^2}E_x+\frac{\partial^2}{\partial z^2}E_x+\left(\left(\frac{\omega}{c}\right)^2-\beta^2\right)E_x&=& 0,\label{TM}\\
	 \frac{\partial^2}{\partial y^2}B_x+\frac{\partial^2}{\partial z^2}B_x+\left(\left(\frac{\omega}{c}\right)^2-\beta^2\right)B_x&=& 0.\label{TE}
	 \end{eqnarray}
	 To the above equations we need to add the boundary conditions at the boundary of the waveguide, $\BE_0\cdot{\bf \tau}=0$ and $\BB_0\cdot{\bf n}=0$ where ${\bf \tau}$, ${\bf n}$ are the tangential and respectively exterior normal to the boundary surface of the waveguide and where we used the fact that the waveguide surface is a perfect conductor.
	 
	 From \eqref{transverse-fields} and \eqref{TM} with \eqref{TE} we have that every wave traveling in the $\Bx$ direction in the waveguide is determined by its longitudinal components, $B_x$ and $E_x$. Thus  the waves propagating in the waveguide can be represented as a  superposition of transverse electric TE ($E_x=0$) and transverse magnetic TM ($B_x=0$) waves. It is also well known that transverse electromagnetic waves TEM, waves with $E_x=0=B_x$ do not propagate in a hollow waveguide. In this paper we will focus only on the study of TM waves.

	 First, after changing to cylindrical coordinates in \eqref{helmholtz-full-2D} one can easily observe that the incoming TM guided wave propagating in the $x$ direction, has the longitudinal component of its electric field $\BE_{x}$ represented by,
	 
	  \begin{figure}[!ht]
	       \centering
	      \includegraphics[width=5in, height=5in]{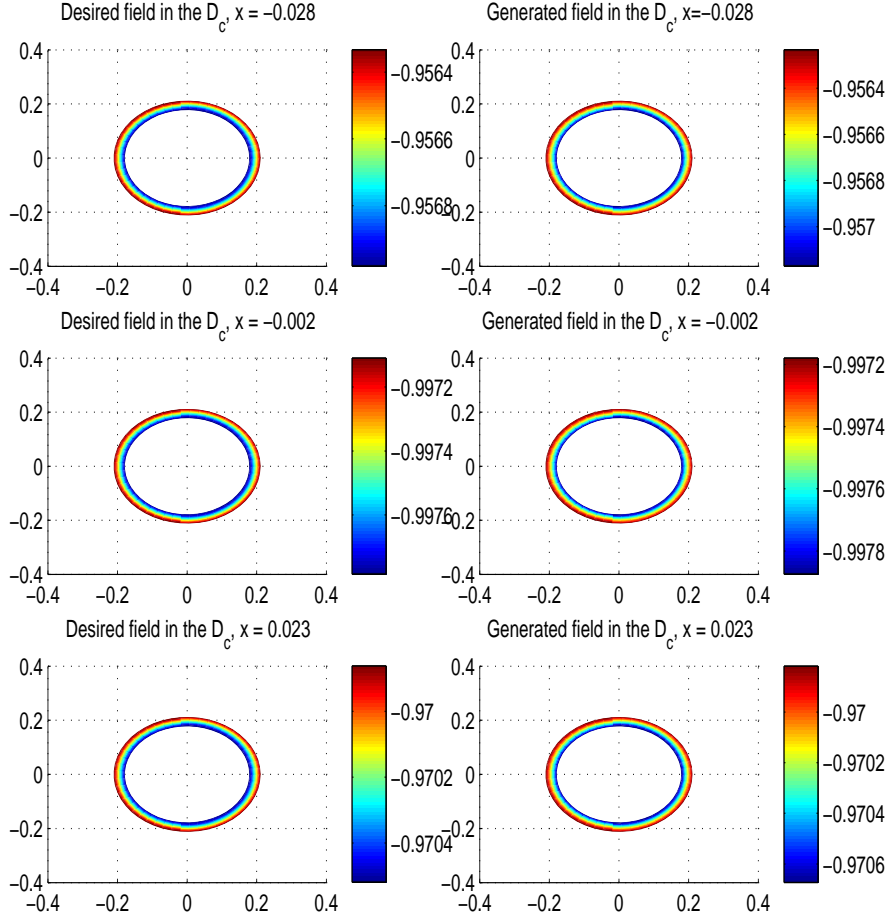}
	       \vspace{0cm}
	       \caption{Control of $TM_{10}$ mode}
	       \label{Fig1}
	   \end{figure}

	 \begin{equation}
	 \label{interogatorTM}
	 \BE_x(r,\theta,x)=\sum_{m,n}E_{mn}J_m(\frac{r\chi_{mn}}{R})cos(m\theta)e^{i\beta_{mn}x},
	 \end{equation}
	  where $E_{mn}$ are coefficients determined by the interrogating field, $J_m$ represents the $m$-th order Bessel-$J$ function, $\chi_{mn}$ represents the $n$-th root of $J_m$, and $\beta_{mn}=\sqrt{\left(\frac{\omega^2}{c^2}-\frac{\chi^2_{mn}}{R^2}
	  \right)}$. For any $m,n$ the term  $E_{mn}J_m(\frac{r\chi_{mn}}{R})cos(m\theta)e^{i\beta_{mn}x}$ is   called the $TM_{mn}$ normal mode and it solves \eqref{TM} with zero boundary data on $\partial B_R(\B0)$, where $\lambda=\lambda_{mn}=\frac{\omega^2}{c^2}-\beta_{mn}^2=\frac{\chi^2_{mn}}{R^2}$. Note that we can chose $D_{a'}$ and $D_1$ in \eqref{3} and \eqref{3'}so that $\lambda_{mn}$ satisfies the hypothesis of Theorem \ref{principala-teorema}. Thus we can use the theoretical results of the previous section.
	  In our numerical results we will show control of the first dominant TM mode with normalized amplitude, i.e., $E_{1x}=J_0(\frac{r\chi_{10}}{R})$. For the numerical experiments we considered a radial frequency of 300Mhz, the waveguide $D$ with $R=5$ and the small cylindrical antenna, $D_a$, with $\delta=0.05$ and $l=0.3$ is placed in it to control the longitudinal component of the electric field in an incoming $TM_{10}$ mode. The control is numerically showed in a near field cylindrical coaxial shell, $(x,r,\theta)\in D_c=[-0.1,0.1]\times [0.13,0.16]\times [0,2\pi]$. \medskip
	  
	  \begin{figure}[!ht]
	         \centering
	          \includegraphics[width=4in, height=4in]{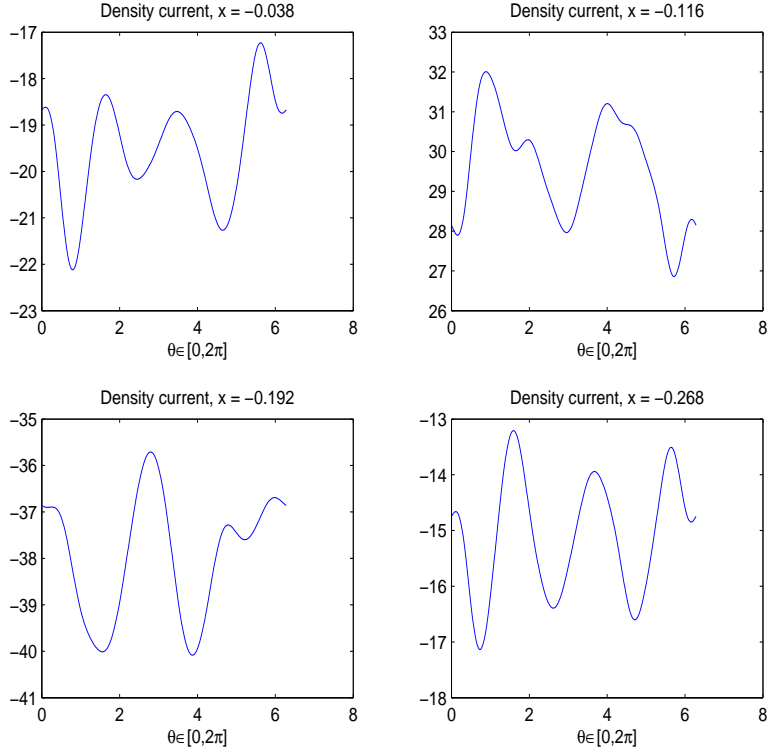}
	         \caption{Density current $v_n$ as function of $\theta$ for four points on the negative $x$ axis}
	         \label{Fig2}
	     \end{figure}
	     
	      \begin{figure}[!ht]
	              \centering
	               \includegraphics[width=4in, height=4in]{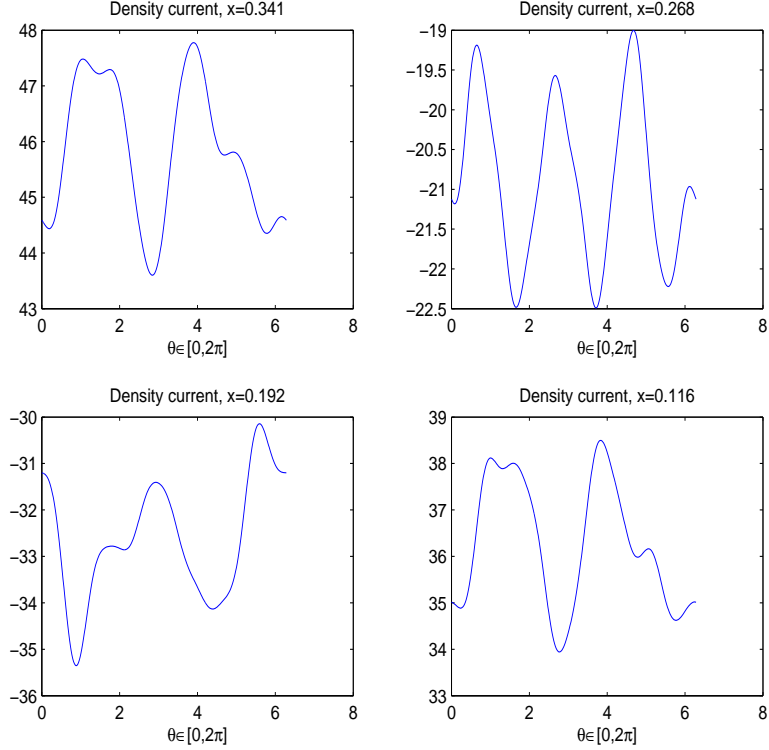}
	               \caption{Density current $v_n$ as function of $\theta$ for four points on the positive $x$ axis}
	              \label{Fig3}
	          \end{figure}

	  Figure \ref{Fig1} describes the contour plots of cross-sections of the desired field $-\BE_x$ (left column) versus the cross-sections of the field generated by the active antenna ( right column). The three rows in Figure \ref{Fig1} represent the two fields in $D_c$ for three cross-section of the antenna, $x=-0.028$, $x=0.002$, $x=0.023$ respectively.

	  Figure \ref{Fig2} and Figure \ref{Fig3} each describe the density $v_n$ introduced in Theorem \ref{principala-teorema} as a function of the angle $\theta\in[0,2\pi]$ for four different values of $x$.  We used the Moment Method and Tikhonov regularization scheme to approximate the density function $v_n$ on several cross-sections around the center of the antenna.
	  
	   One can see the high accuracy of the control as well as the limited power budget needed for the desired control which is not the case in the far field schemes (see \cite{Norris}). 
	     In fact the relative $L^{\infty}$ error between the given $E_{1x}$ and the field generated by our active source is of order $O(10^{-4})$ for the dominant mode $TM_{10}$. 
	   We only presented the mode with the lowest cutoff frequency but mention that by superposition one will be able to control, in theory, the longitudinal component of any incoming TM field.

		\section{Feasibility discussion}
		\label{F}
		In this section we will present the analysis of {\bf Problem 2.} of Question \ref{QUST:Main}. The next result, shows the connection between the class of solutions obtained in the previous section and the main problem stated at Question \ref{QUST:Main}. We have,
		\begin{theorem}
		\label{Thm-feasibility}
		There exists realistic small antennas $D_a$ and small magnetic currents $\BM$ to be instantiated on $\partial D_a$ such that the electromagnetic field $(\BE,\BH)$ generated will have the property that the longitudinal component of its electric field $E_{xs}$ will satisfy
		\begin{equation}
		\label{feas-res}
		\Vert u_n-E_{xs}\Vert_{L^2(\partial D_a)}=o(1),\mbox{ as } n\rightarrow\infty,
		\end{equation}
		Thus, due to Lemma \ref{lemma-2}, $E_{xs}$ will be very close to $u_n$ described in \eqref{defun-0} and hence achieve the desired control in $D_c$ with vanishingly small trace values on the boundary of the waveguide.  The surface magnetic current needed for this is characterized by the property that it  approaches zero near the ends of the antenna (tappered to zero there) and it satisfies the condition $\BM\cdot \hat{\theta}=E_b$ on $\partial D_a$ where $\hat{\theta}$ denotes the unit vector in the $\theta$ direction for a cylindrical coordinate system on $D_a$.
		\end{theorem}
		
	\begin{proof} 
	
For the proof of Theorem \ref{Thm-feasibility} we will need the next Lemma which studies the injectivity of the operator $K$ defined at \eqref{EQ:K}.
		\begin{lemma}
		\label{injectivity-K}
		 Assume the conditions of Theorem \ref{principala-teorema}. Then the operator $K$ is injective.
		\end{lemma}
		
		\begin{proof}
		Consider $w\in L^2(\partial D_{a'})$ such that $Kw=0$. Then we have 
		\beq
		\int_{\partial D_{a'}}w(\By)\frac{\partial \Phi(\Bx,\By)}{\partial \nu_\By}ds_\By=0\mbox{ on }\partial D_1.
			\eeq{i1}
			From  \eqref{ass-lambda} and by using the uniqueness of the interior Dirichlet problem we have that 
			\beq
		\int_{\partial D_{a'}}w(\By)\frac{\partial \Phi(\Bx,\By)}{\partial \nu_\By}ds_\By=0\mbox{ in }{\bar D}_1.
					\eeq{i2}
					By analytic continuation we get 
	\beq\int_{\partial D_{a'}}w(\By)\frac{\partial \Phi(\Bx,\By)}{\partial \nu_\By}ds_\By=0\mbox{ in }D_L\setminus {\bar D}_{a'}.
								\eeq{i3}	
			From \eqref{i3} we have that 
			\beq
		2\int_{\partial D_{a'}}w(\By)\frac{\partial \Phi(\Bx,\By)}{\partial \nu_\By}ds_\By+w(\Bx)=0\mbox{ on }\partial D_{a'}.
							\eeq{i4}
		Next, using our assumption that $-\lambda^2$ is not an interior Neumann eigenvalue for $D_{a'}$, from \eqref{i4} we obtain $w=0$. 
		\end{proof}
		The next Lemma presents a technical regularity result. Since the proof is classical we do not include it here but point out that the result can be obtained by adapting the proof in \cite{Kress-Book99}, (Section 3.4).
		\begin{lemma}
			\label{lemma-2}
			Consider $\lambda\in\RR$ as in Theorem \ref{principala-teorema}. Let $f_e\in H^{\frac{1}{2}}(\partial D_a))$ and let $v_e\in C^2(\RR^3\setminus {\bar{D}_a})\cap C(\RR^3\setminus D_a)$ 
			be the unique solution of the following problem,
			\beq
			\left\{\begin{array}{llll}
			\GD v_e +\lambda^2 v_e=0 \mbox{ in } \RR^3\setminus{\bar{D}}_a\vspace{0.15cm}\\
			v=f_e \mbox{ on } \partial
			D_a\vspace{0.15cm}\\
		\displaystyle\left(\frac{x}{|x|},\nabla v_e(x)) -i\lambda v_e(x)\right)=o\left(\frac{1}{|x|}\right),\mbox{ as }|x|\rightarrow\infty, \vspace{0.15cm}\\
		\mbox{ uniformly for all directions $\displaystyle\frac{x}{|x|}$. }\end{array}\right.
			\eeq{6'}
				Then, for any given open domain $F\subset \RR^3\setminus{\bar D}_a$ we have,
			$$\displaystyle\Vert v_e\Vert_{H^1(F)}\leq
			C\Vert f_e\Vert_{L^2(\partial D_a)},$$
			where $C=C(\lambda,\dist(\partial D_a,\partial F))$ with $\dist$ function denoting the usual distance between two sets. 
			\end{lemma}
		 We will next propose an antenna geometry together with a strategy for the instantiation of a surface magnetic current $\BM$ satisfying the conditions in Theorem \ref{Thm-feasibility}. Remember that the antenna was assumed to be a coaxial cylinder with round edges and small cross-section area. Thus, consider $\Gd_n\in C^2(-l,l)$,  given by $\Gd_n(x)=\mu_n d(x)$ where $\mu_n\rightarrow 0$ as $n\rightarrow\infty$, and $d\in C^2(-l,l)$ is a smooth cut-off function such that there exists a positive real $0\leq c\leq l$ satisfying,
		  \begin{equation} \label{d-d} 0\leq d=\left\{\begin{array}{ll}
		  0, &\mbox{ for } |x|=l\\
		  1, &\mbox{ for } x\in (-l+c,l-c)\end{array}\right.,\end{equation}
		  and $$|d'|<C(l,c), \mbox{ for some constant $C(l,c)$ depending only on $l,c$}. $$
		  Note that, from the definition of $\delta_n$ and \eqref{d-d} we have
		  \begin{equation}
		  \label{deltax}
		  |\delta_n'|\leq C(l,c)\mu_n, \mbox{ for all } n.
		  \end{equation} 
		  Assume the geometry of the antenna to be described by $D_a=[-l,l]\times B_{\delta_n(\B0)}(\Bx)$ for some $l>0\in\RR^3$ for some suitable $n$ large enough. 
		  Then $\Bz\in\partial D_a$ is given by, $\Bz=(x,\delta_n(x)\cos(\theta),\delta_n(x)\sin(\theta))$ for $\theta\in[0,2\pi]$ and $\partial D_a$ is a $C^2$ smooth surface of rotation with local coordinates $(x,\theta)$ and local tangential vectors $(\frac{d{\bf z}}{dx},\frac{d{\bf z}}{d\theta})$.

The following result is a consequence of Lemma \ref{injectivity-K}, classical Fredholm Theory and Tikhonov regularization arguments (see Theorems 3.15,3.17 in \cite{CoKr-Book83} and Theorem 4.13 in \cite{CoKr-Book98} for example),
\begin{corollary}
		\label{reg-bound}
		The boundary data $E_b$ given in Theorem \ref{principala-teorema} satisfies,
		$$\Vert E_b\Vert_{C^2(\partial D_a)}=\frac{C \sqrt{\mu_n}}{s_n^3}\Vert u_n\Vert _{L^2(\partial D_{a'})}\leq \frac{C \sqrt{\mu_n}}{s_n^3\sqrt\alpha_n} \Vert E_{xi}\Vert _{L^2(\partial D_1)},\mbox{ for every }n,$$
		\end{corollary}
		where $s_n=dist(\partial D_{a'},\partial D_a)$. Observe that for suitable $s_n$ (which can be chosen with a suitable choice of $D_{a'}$ satisfying \eqref{3'}) and $\alpha_n$ we have that 
		  \begin{equation}
		  \label{estimare}
		  \frac{C \sqrt{\mu_n}}{s_n^3\sqrt\alpha_n}=o(1), \mbox{ as } n\rightarrow\infty.
		  \end{equation}
		  
		  Let $\BM$ be such that $\BM=E_b(x,\theta)\hat{\theta}$ 
		where $E_b$ is given at \eqref{un}.  Note that estimate \eqref{estimare} and Corollary \ref{reg-bound} indicate that,in the limit when $n\rightarrow \infty$, the current $\BM$ will satisfy the continuity equation. Also, note that $v_n$ introduced at \eqref{reg} belongs to $C^{\infty}(\partial D_a)$. From this, Corollary \ref{reg-bound} we obtain that the vector field $\BM$ will be smooth enough so that for an antenna chosen so that $\lambda$ is not an interior Maxwell eigenvalue the following exterior Maxwell problem problem admits a unique solution (see \cite{CoKr-Book83}, Section 4.4):
		
		\vspace{0.3cm}
		Find $\BE=(E_x,E_y,E_z),\BH=(B_x,B_y,B_z) \in C^1(\RR^3\setminus{{\bar D}_a})\cup C(\RR^3\setminus D_a)$ such  that
		\begin{equation}
		\label{ext-M}
		\left\{\begin{array}{llll}
		\nabla\times\BE-i\lambda \BH=0, &\;\nabla\times\BH+i\lambda \BE=0\mbox{ in }\RR^3\setminus{{\bar D}_a} \vspace{0.2cm}\\
		-\nu\times\BE=\BM,& \mbox{ on }\partial D_a\vspace{0.2cm}\\
		\left[\BH,\frac{\Bx}{|\Bx|}\right]-\BE=o(\frac{1}{|\Bx|}),& \mbox{ as }|\Bx|\rightarrow\infty\vspace{0.2cm}\\
		\left[\BE,\frac{\Bx}{|\Bx|}\right]+\BH=o(\frac{1}{|\Bx|}),& \mbox{ as }|\Bx|\rightarrow\infty,\vspace{0.2cm}\end{array}\right.\end{equation}
		where $\nu$ is the unit exterior normal to $D_a$ and where $\BM$ is as above. After simple algebraic manipulations of the boundary condition in \eqref{ext-M} we observe that the first component of the electric field $\BE$ solution of \eqref{ext-M}, namely $E_{xs}$, is given by,
		\begin{equation}
		\label{Ex}
		E_{xs}(x,\theta)= E_b(x,\theta)\sqrt{(\delta'_n(x))^2+1}-\delta'_n(x)(E_z \sin(\theta)+E_y \cos(\theta)).
		\end{equation}
		Using \eqref{Ex} we obtain
		\begin{eqnarray}
		\Vert E_{xs}-E_b(x,\theta)\Vert_{L^2(\partial D_a)}&\leq& \left\Vert \frac{(\delta'_n(x))^2 E_b(x,\theta)}{\sqrt{(\delta'_n(x))^2+1}+1}\right\Vert_{L^2(\partial D_a)}\!\!\!\!+\Vert \delta'_n(x)(E_z \sin(\theta)+E_y \cos(\theta))\Vert_{L^2(\partial D_a)}\nonumber\\
		&\leq& o(1), \mbox{ as }n\rightarrow\infty
		\end{eqnarray}
		In the second identity of \eqref{Ex} we have used \eqref{deltax}, \eqref{estimare} and the fact that $\Vert E_z\Vert_{L^2(\partial D_a)}$ and $\Vert E_y\Vert_{L^2(\partial D_a)}$ are $o(1)$ on $\partial D_a$ as $n\rightarrow\infty$. Behavior of $\Vert E_z\Vert_{L^2(\partial D_a)}$ and $\Vert E_y\Vert_{L^2(\partial D_a)}$ can be obtained by observing that from Corollary \ref{reg-bound} we have $\Vert \BM\Vert_{L^2(\partial D_a)}\rightarrow 0$ on $\partial D_a$ as $n\rightarrow \infty$ and then by employing the continuity of the solution operator for problem \eqref{ext-M} with $L^2$ boundary data (see \cite{Angel-Kirsch}, Section 6.3). In fact this behavior can also be obtained from the observation that based on our  strategy of a possible instantiation presented below we need small conduction currents to excite the small magnetic current $\BM$ and thus the electric field will be very small on the surface of the antenna.
		
		Next by using Lemma \ref{lemma-2} we obtain that for suitable chosen $\alpha_n$ there exist $n$ large enough such that for any sub-domain $F$, $F\subsetneq \RR^3\setminus{\bar D}_a$, we have
		\begin{equation}
		\label{aprox-00}
		\Vert E_{xs}-E_b\Vert_{H^1(F)}\leq Co((\mu_n)^\frac{1}{2}),\mbox{ as }n\rightarrow\infty.
		\end{equation}
		where $C=C(\lambda,\dist(\partial F, \partial D_a))$ and we have used that $|\partial D_a|=O(\mu_n)$. 
		
		Thus, we showed that for large enough $n$, an antenna with the geometry described above and with a magnetic current $\BM$ excited on its surface as indicated above generates an electromagnetic field $\BE,\BH$, solution of \eqref{ext-M}, which  satisfies the properties of Theorem \ref{Thm-feasibility}. 
		\end{proof}
		
			\begin{figure}[!ht]
					       \centering
					       \includegraphics[width=4in, height=3in]{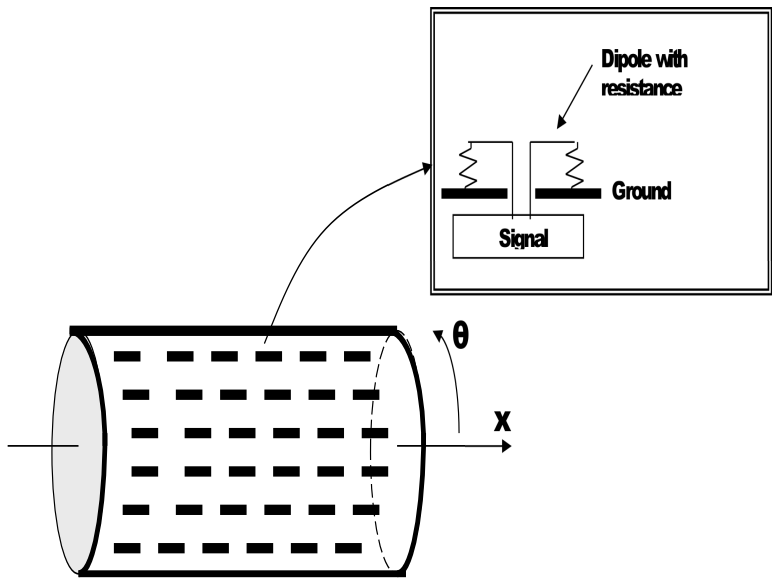}
					       \caption{The proposed antenna is a closely spaced array on a metallic cylinder with small diameter.  Each dipole (as shown in the inset) is resistively loaded to the metallic cylinder so the dipole currents are spatially constant across the length of the dipole.    }
					       \label{Fig-Instantiation}
					   \end{figure}
				\vspace{0.5cm}	  	   
					   
The central purpose of this paper is to study the process whereby electromagnetic fields can be nulled or otherwise controlled locally in the sinusoidal or harmonic case.  The infinite circular waveguide was studied because of the stability of its modes and their simplicity relative to other structures including far field regimes. 

However, there also is a practical interest.  On one hand, devices to switch signals from one waveguide to another in a branching configuration are commercially available and based on our review of the literature and practice these devices appear to be mechanical in nature.  Additionally, there has also been consideration of suppression of unwanted modes ( references \cite{*},\cite{**}). Thus, this research on suppression of longitudinal modes in a cylindrical waveguide may have some interest to microwave engineers for mode suppression or switching.   This leads to a consideration of physical realization of the control activity specified by the analysis presented above.

		We will now discuss about a possible instantiation of this antenna (see Figure \ref{Fig-Instantiation}). As discussed above, the sufficient condition that the longitudinal component of the electric field of the antenna (modeled as magnetic surface currents $\BM$) has the desired control properties in the near field control region $D_c$ with vanishingly small values on $\partial D$, is that the inner product of the local magnetic current $M(x,\theta)$ with the unit vector $\hat{\theta}$ equals $E_b$ described \eqref{un-0}.  Therefore from the point of view of realizing this nulling antenna this sufficient condition must be instantiated. The complexity of the problem is well illustrated by reference to Figures \ref{Fig2} and \ref{Fig3} above.  It is evident that the field control is made possible in this analysis by a boundary density function that is a function of distance along the longitudinal axis of the central antenna, and is also an independent function of the angle theta around the axis.  The controlled modes are modes that are parallel to the antenna surface.  In the electromagnetic case, the question arises of how this would be accomplished.
		      
		      In the case of exterior scattering problems, Kersten (\cite{+}) has shown that boundary densities may be replaced for a set of interior electric and magnetic dipoles while preserving the exterior scattered field (see also  \cite{++}   ).  While an argument equivalent in rigor is not known to us concerning antenna field generation by secondary sources, the method of auxiliary sources has been used effectively in the assessment of antenna radiation properties including far field structures and input impedance ( \cite{=}, \cite{==}, \cite{===}).   Thus, we have been encouraged to consider an electromagnetic physical realization of the complex radiating structure called for in the analysis described above.  Also, perhaps most importantly, Theorem 5.1 translates the Helmholtz control results and calculations of Sections \ref{SEC:Form}, \ref{SEC:Exist} , and \ref{SEC:Num} into electromagnetic terms indicating the existence of a suitable antenna. 
	
		 A candidate structure which we are studying to realize the requirements of this analysis is an array of very small horizontally oriented dipoles supported on a conductive surface as diagrammed in the Figure \ref{Fig-Instantiation}. Since the dipoles are resistively loaded (that is, there is a conductive path to the cylindrical ground plane), a spatially constant, temporally oscillating current can be maintained in the metallic arms of each dipole using the signal generator.  Thus, this arrangement permits a varying current with theta and with the linear dimension x.  This is an example of a dense array, an antenna design of present interest (\cite{^}).  This surface current developed by the array is such that it is proportional to the electric field parallel to the cylinder (a longitudinal) mode, and thus the current acts as the desired magnetic current density.

		 \begin{figure}[!htp]
		 					       \centering
		 					       \includegraphics[width=5in, height=4in]{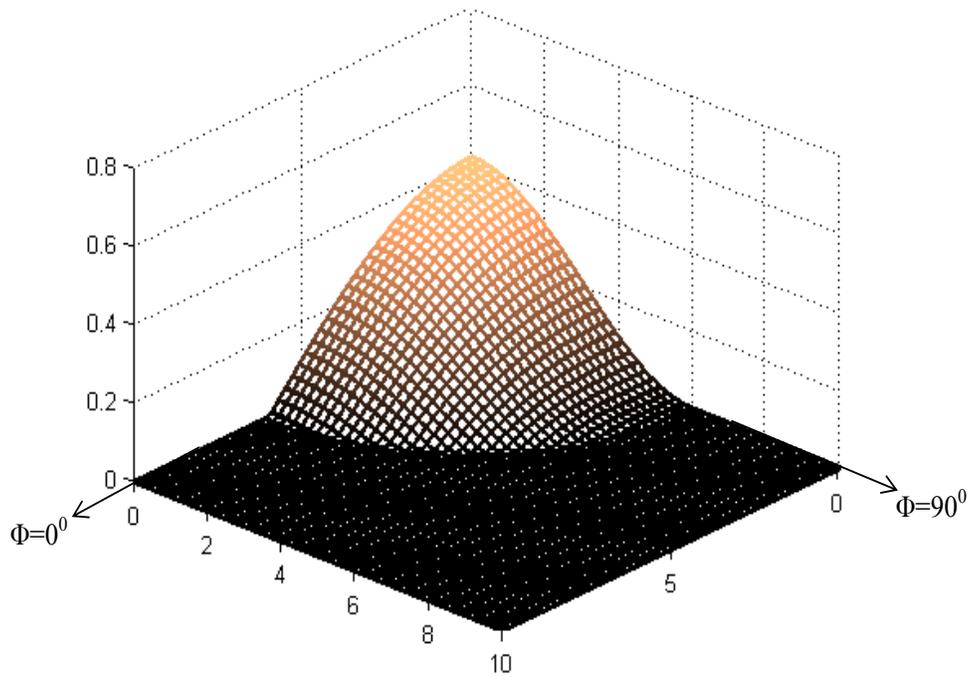}
		 					       \caption{This figure shows an estimate of the dipole field amplitude in the direction parallel to the cylindrical array structure as a function of elevation (polar) angle and azimuth angle around the horizontal dipole (in the figure placed in the north corner of the $x,y$ plane  along the $\Phi=0°$  axis) .  Each plane in this figure orthogonal to the $x,y$ plane and including the z axis, provides a view of the field amplitude with elevation angle  at a given azimuth.  The  $\Phi=0°$ plane corresponds to the zero azimuth angle. }
		 					       \label{Fig-dipole}
		 					   \end{figure}
		 \vspace{0.5cm}
		 
		 Each dipole produces a field that has a component that is directed radially (orthogonal to the cylinder surface), but most importantly, as stated in the paragraph above, there is a component that is parallel to the cylindrical surface.  The field structure of a horizontal dipole over a perfect ground plane is shown in Figure \ref{Fig-dipole} following Balanis’ development(\cite{B}).  This field structure approximates the field of the dipole mounted on a cylindrical metallic structure as shown above.  Combining these field structures using array principles results (treating each field structure as a basis function) in an approximation to the effect of the controlling current density specified in this Section.


	\section*{Acknowledgment}
	The Authors would like to acknowledge the AFOSR support of this work under the 2013 YIP Award FA9550-13-1-0078.

	
	
	\end{document}